\documentclass[reqno,11pt]{amsart}

\usepackage{labelfig}
\usepackage{epsfig}
\usepackage{epstopdf}

\usepackage{pinlabel,cite}
\input{epsf.tex}
\usepackage{latexsym,amsfonts,amssymb,verbatim,mathrsfs,amsthm}
\usepackage{amsmath,amsthm,amssymb,latexsym,graphics,textcomp}
\usepackage{eucal,eufrak}
\usepackage{colonequals}
\usepackage{graphicx}
\usepackage{url}
\input{xy}
\xyoption{all}

\theoremstyle{plain}
\newtheorem{theorem}{Theorem}[section]
\newtheorem{lemma}[theorem]{Lemma}
\newtheorem{corollary}[theorem]{Corollary}

\newtheorem{proposition}[theorem]{Proposition}

\newtheorem*{claim}{Claim}

\theoremstyle{definition}
\newtheorem*{remark}{Remark}

\def\1{\mathbf 1}

\def\F{\mathcal{F}}

\def\Mod{\mathrm{Mod}}

\def\A{\mathcal A}

\def\F{\mathcal F}

\def\L{\mathcal L}

\def\link{\mathrm{link}}

\title{Injective maps between flip graphs}
\date{\today}
\begin{document}

\author{Javier Aramayona}
\address{Institut de Math\'ematiques de Toulouse, Universit\'e Paul Sabatier, 118 Route de Narbonne, 31062 Toulouse, Cedex 4, France}
\email{Javier.Aramayona@math.univ-toulouse.fr}

\author{Thomas Koberda}
\address{Department of Mathematics, Yale University, 20 Hillhouse Ave, New Haven, CT 06520, USA}
\email{thomas.koberda@gmail.com}

\author{Hugo Parlier}
\address{Department of Mathematics, University of Fribourg, Chemin du Mus\'ee 23, CH--1700 Fribourg, Switzerland}
\email{hugo.parlier@unifr.ch}

\maketitle

\begin{abstract}
We prove that  every injective simplicial map  $\F(S) \to \F(S')$ between flip graphs is induced by a subsurface inclusion $S\to S'$, except in finitely many cases. This extends a result of Korkmaz--Papadopoulos which asserts that every automorphism of the flip graph of a surface without boundary is induced by a surface homeomorphism. \end{abstract}

\section{Introduction}

Consider a compact, connected and orientable surface $S$, of genus $g\ge 0$ with $b\ge 0$ boundary components. Moreover, assume that $S$ has $p+q>0$ marked points, with $p\ge 0$  in the interior of $S$ and the other $q\ge 0$ in $\partial S$, subject to the condition that every component of $\partial S$ must contain at least one marked point.  When convenient, we sometimes think of marked points in the interior of $S$ as punctures.

By an {\em arc} on $S$ we will mean the homotopy class (relative to the marked points) of an arc properly contained in $S$, and which intersects the set of marked points only at its endpoints.  A {\em multiarc} is a collection of arcs on $S$  with pairwise disjoint interiors. A maximal multiarc is called a {\em triangulation}; observe that every triangulation of $S$ contains exactly  $d(S) = 6g + 3b + 3p + q - 6$ arcs. 
The {\em flip graph}  $\F(S)$ is the simplicial graph whose vertices are triangulations of $S$, and where two triangulations are adjacent  if and only if they share exactly $d(S)-1$ arcs; note this implies that the remaining two arcs intersect exactly once.  The reader may wish to note the equivalence of the flip graph and the \emph{triangulation graph} used by other authors, such as M. Bell \cite{Bell}.

Observe that $\F(S)$ is locally finite, as every vertex has valence at most $d(S)$. Since the mapping class group $\Mod(S)$ acts on $\F(S)$ by automorphisms and since the quotient is compact, it follows by the \v{S}varc--Milnor Lemma (see, e.g., \cite{BH}) that $\F(S)$ and $\Mod(S)$ are quasi-isometric. This fact has been exploited by Disarlo--Parlier \cite{DP}
to give an elementary proof of a result of Masur--Minsky \cite{MM} that subsurface inclusions induce quasi-isometric embeddings between the corresponding mapping class groups. In a different direction, the flip graph has recently been used by Costantino--Martelli \cite{CM} to construct families of quantum representations of mapping class groups. 

In this paper, we  classify all injective simplicial maps between flip graphs. Before giving a precise statement we need some definitions. Given surfaces $S$ and $S'$, by an {\em embedding} of $S$ into $S'$ we mean a  $\pi_1$--injective continuous map $h:S \to S'$ that maps every marked point on $S$ to a marked point on $S'$. 

An embedding $h:S\to S'$ induces an injective simplicial map $\phi:\F(S) \to \F(S')$ as follows: we choose a triangulation $A$ of $S' \setminus\text{int}(h(S))$, plus a collection $B$ of arcs on $\partial h(S)$ whose union is homeomorphic to $\partial h(S)$,  and then define $\phi(v) = h(v) \cup A \cup B$ for all $v\in \F(S)$. 

The purpose of this paper is to prove that, provided $S$ is ``complicated enough",  every injective simplicial map $\F(S) \to \F(S')$ arises in this way. More concretely, we will say that the surface $S$ is {\em exceptional} if it is an essential  subsurface of (and possibly equal to)  a torus with at most two marked points, or a sphere with at most four marked points. 
Our main result is:

\begin{theorem} 
Suppose $S$ is non-exceptional. Then every injective simplicial map  $\phi: \F(S) \to \F(S')$ is induced by an embedding $S\to S'$.
\label{main}
\end{theorem}

Note that if $S$ is a cylinder with two boundary components and one vertex on each boundary then $\F(S) \cong \mathbb R$ with its usual simplicial structure, and thus the statement of Theorem \ref{main} is false in this case. If $S$ is a torus with one marked point then $\F(S)$ is an infinite trivalent tree, which we conjecture can be embedded in the flip graph of any surface of genus $\ge 2$ with one marked point in a way that is not induced by an embedding between the corresponding surfaces. While these examples highlight the failure of Theorem \ref{main} for an arbitrary surface $S$, we do not know at this time whether Theorem \ref{main} holds for some of the surfaces excluded in the hypotheses. 

Theorem \ref{main} should be compared with a previous result \cite{Ara} of the first author, which shows the analogous statement for injective maps between pants graphs
of surfaces. While the proofs of both results are similar in spirit, the technicalities are rather different; that said, we suspect that it should be possible to give axiomatic conditions for certain classes of graphs, built from arcs or curves on surfaces, that guarantee that simplicial injections between two such graphs are always induced by embeddings of the underlying surfaces.

As usual, the flip graph becomes a geodesic metric space by declaring the length of each edge to be equal to 1. Combining Theorem \ref{main} with the convexity result of  Disarlo--Parlier\cite{DP}, we obtain the following corollary:

\begin{corollary}
Suppose $S$ is non-exceptional, and let  $\phi: \F(S) \to \F(S')$ be an injective simplicial map. Then $\phi(\F(S))$ is a totally geodesic subset of $\F(S')$; in other words, any geodesic in $\F(S')$ connecting two points of $\phi(\F(S))$ is entirely contained in $\phi(\F(S))$. 
\end{corollary}

We now give an idea of the proof of Theorem \ref{main}. The first step will be to show the following:

\begin{theorem}
Suppose $S$ is non-exceptional, and let $\phi: \F(S) \to \F(S')$ be an injective simplicial map. Then:
\begin{enumerate}
\item We have $d(S) \le d(S')$;
\item There exists a multiarc $A \subset S'$, with $d(S') - d(S)$ elements, such that $A \subset \phi(v)$ for all $v \in \F(S)$. 
\end{enumerate} 
\label{invariant-arc}
\end{theorem}

In other words, $\phi(\F(S)) \subset \F_A(S')$, where $\F_A(S')$ denotes the subgraph of $\F(S')$ spanned by those triangulations of $S'$ that contain $A$. Observe that there is a natural isomorphism $\F_A(S') \cong \F(S' \setminus A)$, where $S'\setminus A$ is the result of cutting $S'$ open along  every element of $A$, and
thus we can view the map $\phi$ as a simplicial injection $ \F(S) \to \F(S' \setminus A)$. Noting that $d(S) = d(S'\setminus A)$,
Theorem \ref{main}  will follow from:

\begin{theorem}
Let $S$ and $S'$ be connected surfaces, with $d(S) =d(S')$ and $S$ non-exceptional.  Then every injective simplicial map $\F(S) \to \F(S')$ is induced by a homeomorphism $S \to S'$.
\label{isom}
\end{theorem}

\begin{remark} Theorem \ref{main} was previously shown by Korkmaz-Papadopoulos \cite{KP} in the special case when $S = S'$, $\partial S =\emptyset$, and the map $\F(S)\to \F(S')$ an automorphism.
\end{remark}

\medskip

\noindent{\bf Acknowledgements.} The authors would like to thank the Institute for Mathematical Sciences of Singapore,  the Mathematisches Forschungsinstitut Oberwolfach, the Technion, and the Universidad de Zaragoza, where parts of this work were completed.  We would like to thank the referee for his/her comments.

The first named author was supported by BQR and Campus Iberus grants. The second named author was partially supported by NSF grant DMS-1203964. The third author is supported by Swiss National Science Foundation grant PP00P2-128557.


\section{Paths in $\F(S)$} 
Similarly to the case of pants graphs \cite{Ara}, a large part of our arguments boil down to understanding when it is possible to extend a pair of adjacent edges in $\F(S)$ to a
{\em square} or a {\em pentagon}; see below for definitions. As it turns out, this issue is significantly more subtle here than for pants graphs, due to the fact that there are vertices in $\F(S)$ with non-isomorphic links. The purpose of this section is to prove a series of technical results that will overcome these difficulties. We begin with some definitions.

\subsection{Flippable vs. unflippable arcs} Let $v\in \F(S)$ be a triangulation, and let $a\subset S$ be an arc such that $a\in v$. We  will say that $a$ is {\em flippable with respect to $v$} if  there exists a triangulation $v' \in \F(S)$ that is adjacent to $v'$ in $\F(S)$ and satisfies $v \setminus (v \cap v') =a$.   In other words, the edges in the triangulations $v$ and $v'$ of $S$ differ only by the arc $a$.  We will denote the flip from $v$ to $v'$ by $a \to a'$. Observe that  $a\in v$ is unflippable if and only if, up to a homeomorphism of $S$, $v$ contains the arcs in Figure \ref{f:nonflip}; furthermore, if $v'$ is any other triangulation containing those arcs, then $a$ is unflippable with respect to $v'$.

\begin{figure}[h]
\leavevmode \SetLabels
\L(.44*.99) $b$\\
\L(.53*.54) $a$\\
\endSetLabels
\begin{center}
\AffixLabels{\centerline{\epsfig{file =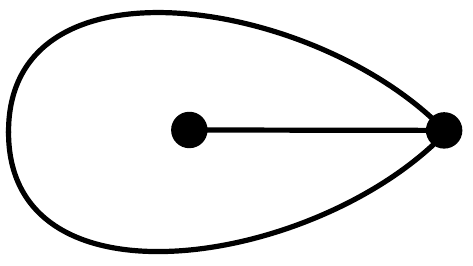,width=4.0cm,angle=0} }}
\vspace{-10pt}
\end{center}
\caption{An unflippable arc of a triangulation} \label{f:nonflip}
\end{figure}

Given a vertex $v\in \F(S)$, denote by $\deg(v)$ the {\em valence} of $v$ in $\F(S)$; that is, the number of vertices of $\F(S)$ that are adjacent to $v$.  Observe that $\deg(v) \le d(S)$ for every vertex $v\in \F(S)$, and that there exists a vertex $u\in \F(S)$ for which $\deg(u) = d(S)$; indeed, in the light of the previous paragraph it suffices to consider a triangulation that contains no arcs bounding a once-punctured disk.

\subsection{Squares and pentagons} A {\em square} (resp. a {\em pentagon}) in $\F(S)$ is a closed path with four  (resp.  five) vertices. Korkmaz-Papadopoulos \cite{KP} have
shown that every square and pentagon in $\F(S)$ is of the form of the one described in Figures \ref{f:square} and \ref{f:pentagon}; see Lemmas 2.2 and 2.3 of \cite{KP}.  In particular we have the following observation, which we state as a separate lemma.

\begin{lemma}
Let $\sigma$ be a square or pentagon in $\F(S)$. Then $\bigcap_{v\in \sigma} v$ consists of exactly $d(S) - 2$ curves. 
\label{shallow}
\end{lemma}

\begin{figure}[h]
\leavevmode \SetLabels
\endSetLabels
\begin{center}
\AffixLabels{\centerline{\epsfig{file =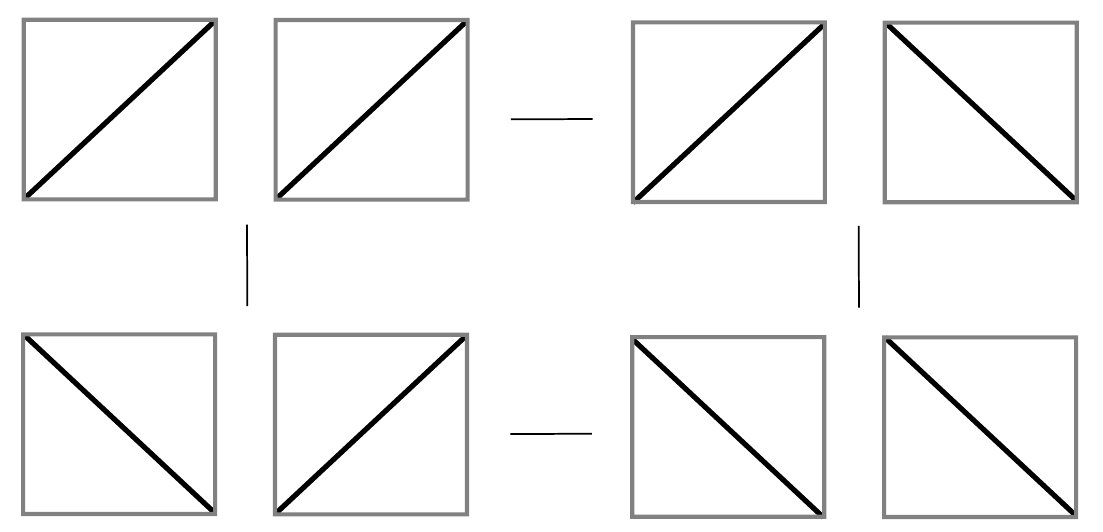,width=6.0cm,angle=0} }}
\vspace{-10pt}
\end{center}
\caption{Schematics of a square in the flip graph} \label{f:square}
\end{figure}

\begin{figure}[h]
\leavevmode \SetLabels
\endSetLabels
\begin{center}
\AffixLabels{\centerline{\epsfig{file =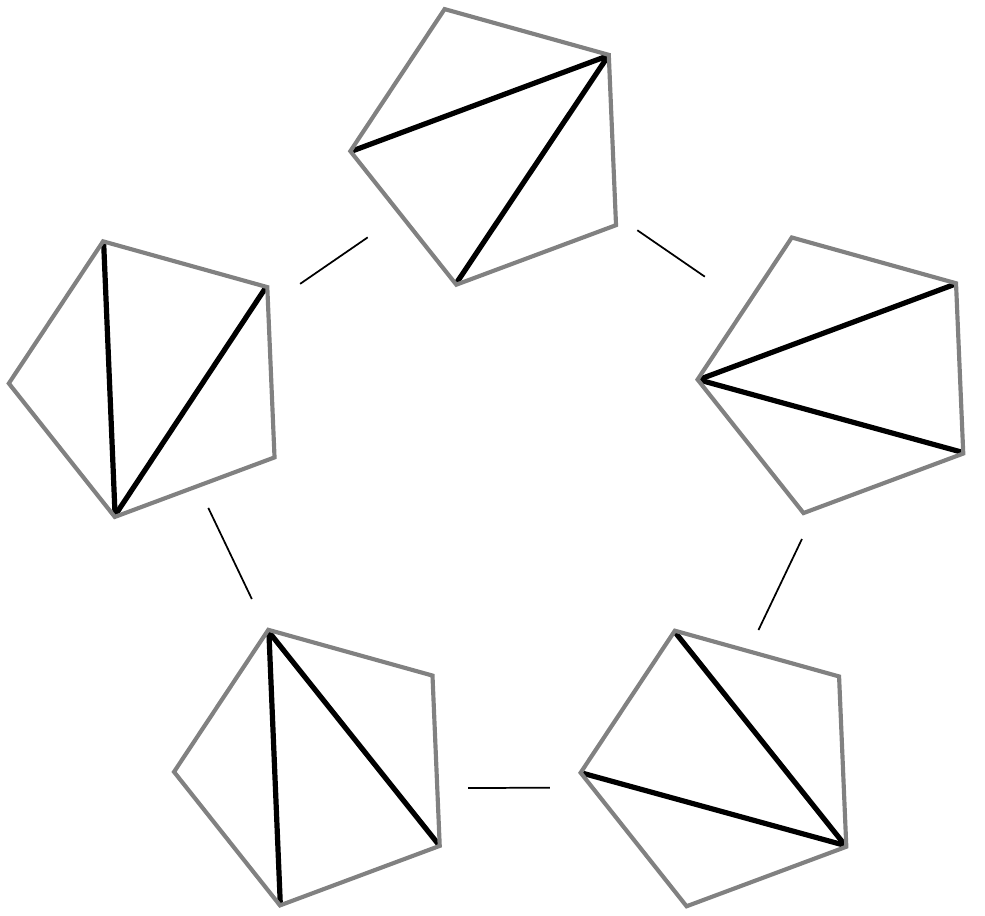,width=6.0cm,angle=0} }}
\vspace{-10pt}
\end{center}
\caption{Schematics of a pentagon in the flip graph} \label{f:pentagon}
\end{figure}

As mentioned above, one of the main difficulties stems from the fact that there exist pairs of adjacent edges of $\F(S)$ that are not contained in a square or a pentagon. In light of this, we  introduce the notion of an {\em extendable} edge: we will say that the (oriented) edge $(u,v)$ of $\F(S)$ is  extendable if for all $w \in \link(v) \setminus \{u\}$, there is a square or a pentagon in $\F(S)$ that contains $\{u,v,w\}$; here, $\link(v)$ denotes the link of the vertex $v$ in $\F(S)$, i.e. the set of vertices adjacent to $v$.  A path is called \emph{extendable} if it consists of extendable edges.

The next two propositions will be crucial for the proof of Theorem \ref{invariant-arc}, as they guarantee that there are enough extendable paths in $\F(S)$.
\begin{proposition}\label{extend1}
Every pair of vertices $u, v\in \F(S)$ with 
$$\deg(u)=\deg(v)=d(S)$$
may be joined by a path $u=u_0,u_1, \ldots, u_n=v$ such that  $(u_i,u_{i+1})$ is extendable for all $0\le i< n$.
\end{proposition}

\begin{proposition}\label{extend2}
Let $u, v\in \F(S)$ be adjacent vertices with $\deg(u)>\deg(v)$. For every $w \in \link(v) \setminus \{u\}$ either $u,v,w$ belong to a common square or there is an extendable path between $u$ and $w$.
\end{proposition}

The proof of the above propositions in somewhat involved and will be broken down into a series of intermediate lemmas. We need some notation before commencing. By a {\em triangle} on $S$ we mean a simply connected region bounded by three arcs on $S$. We define {\em quadrilaterals}, {\em pentagons} and {\em hexagons} in a similar fashion. Finally, by a {\em cylinder} we mean an essential subsurface of $S$ that is homeomorphic to $\mathbb S^1 \times [0,1]$, without any interior marked points, and whose boundary components are two distinct arcs on $S$.

\begin{lemma}\label{lem:degree}
Let $(u,v)$ and $(v,w)$ be adjacent edges in $\F(S)$ with 
$$
\deg(u) \geq \deg(v).
$$
Then either $(u,v)$ and $(v,w)$ are contained in a common square or a pentagon, or the two corresponding flips from $u$ to $v$ and from $v$ to $w$ are supported inside a common cylinder.
\end{lemma}
\begin{proof}
Denote by  $a \to a'$ and $b \to b'$ the flips from $u$ to $v$ and from $v$ to $w$, respectively, and observe that $b\ne a'$. If $a'$ and $b$ belong to two different triangles of $v$ then  $u,v,w$ are contained in a square in $\F(S)$, which has $v':= (u \setminus b) \cup b'$ has remaining vertex. 

Suppose now that $a'$ and $b$ belong to the same triangle of $v$. In this case $b$ is contained in a quadrilateral containing $a$ and $a'$. We denote the remaining sides of this square by $c,d,e$ as in  Figure \ref{f:quadins}. Note that $\deg(u)\ge \deg(v)$ implies that the arc $a$ does not bound a once-punctured disk; otherwise an unflippable arc in $u$ would become flippable in $v$. In turn, this yields that $c\ne d$ and $b \ne e$. 

If $b\ne c$ then there is a second triangle of $v$ to which $b$ belongs. The union of this triangle with the former quadrilateral is a pentagon, and hence one can extend the edges $(u,v)$ and $(v,w)$ to a pentagon in $\F(S)$ as in Figure \ref{f:figure2}. If, on the other hand, $b=c$, then we see that $u,v,w$ are connected by two flips that are supported in a common cylinder. This completes the proof of the lemma.
\end{proof}

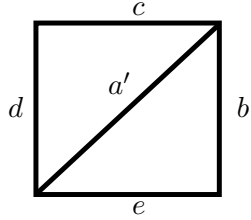
\begin{figure}[h]
\leavevmode \SetLabels
\L(.61*.5) $b$\\
\L(.5*.95) $c$\\
\L(.37*.5) $d$\\
\L(.5*.07) $e$\\
\L(.475*.60) $a'$\\
\endSetLabels
\begin{center}
\AffixLabels{\centerline{\epsfig{file =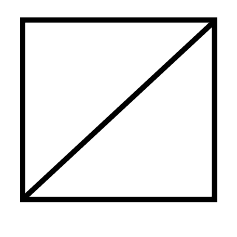,width=3.0cm,angle=0} }}
\vspace{-10pt}
\end{center}
\caption{A quadrilateral in $S$}\label{f:quadins}
\end{figure}




%
%

\begin{figure}[h]
\leavevmode \SetLabels
\L(.324*.5) $b$\\
\L(.524*.78) $b$\\
\L(.47*.91) $a'$\\
\L(.255*.57) $a$\\
\endSetLabels
\begin{center}
\AffixLabels{\centerline{\epsfig{file =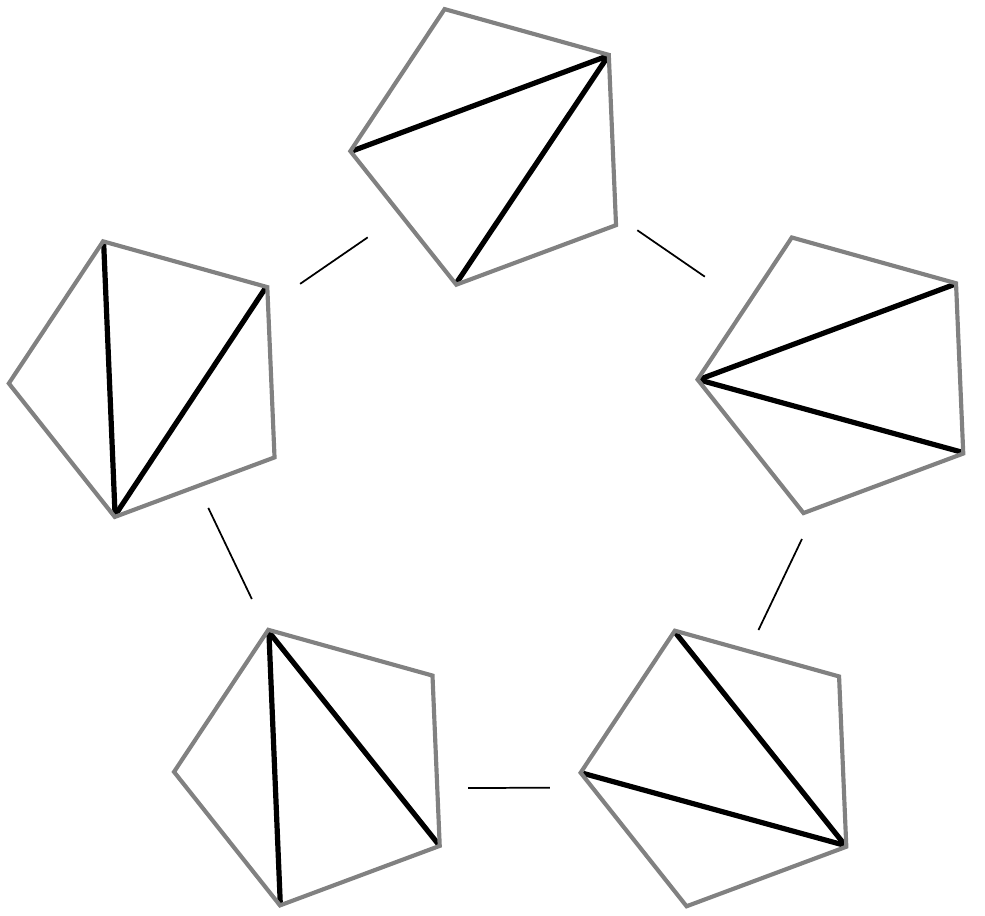,width=7.0cm,angle=0} }}
\vspace{-10pt}
\end{center}
 \caption{Completing edges to a pentagon in $\F(S)$} 
\label{f:figure2}
\end{figure}

%
%

A flip that is supported on a cylinder will be called a {\em cylinder flip}, see Figure \ref{f:figurea}. The following lemma states that it is always possible to ``bypass" a cylinder flip:

\begin{lemma}\label{lem:cylinder}
Let $S$ be a non-exceptional surface and let $u,v \in \F(S)$ be adjacent vertices such that $\deg(u) = \deg(v)$ and the flip from  $u$ to $v$ is a cylinder flip. Then there exists a path $u=u_1,u_2,\hdots,u_k=v$ with 
$$
\deg(u_i)=\deg(u)=\deg(v)
$$
and such that the flip from $u_i$ to $u_{i+1}$ is a non-cylinder flip for all $i\in\{1,\hdots,k-1\}$.
\end{lemma}

\begin{remark}
If $u,v\in \F(S)$ are adjacent vertices of equal degree and  the flip from $u$ to $v$ is not a cylinder flip, then such flip is supported on a quadrilateral on $S$ whose boundary arcs are pairwise distinct; compare with Figures \ref{f:nonflip} and \ref{f:figurea}.
\end{remark}


\begin{proof}[Proof of Lemma \ref{lem:cylinder}]
Let $u$ and $v$ be adjacent vertices with $\deg(u) =\deg(v)$ and such that the flip from $u$ to $v$ is a cylinder flip. Then $u$ and $v$ contain the arcs the left and right pictures of Figure \ref{f:figurea}, respectively; we will refer to the labeling shown therein. 

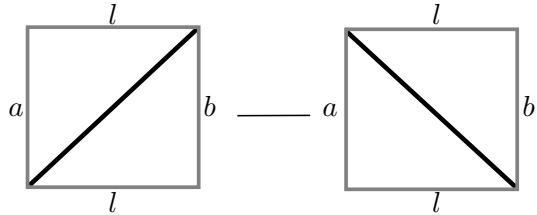
\begin{figure}[h]
\leavevmode \SetLabels
\L(.22*.48) $a$\\
\L(.55*.48) $a$\\
\L(.425*.48) $b$\\
\L(.76*.48) $b$\\
\L(.325*.935) $l$\\
\L(.665*.935) $l$\\
\L(.325*.005) $l$\\
\L(.665*.005) $l$\\
\endSetLabels
\begin{center}
\AffixLabels{\centerline{\epsfig{file =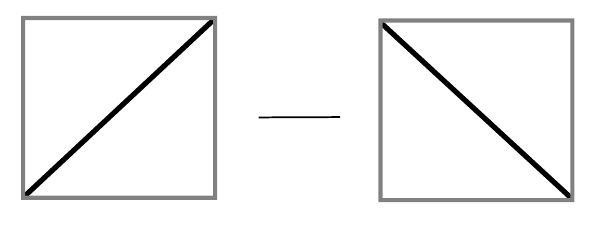,width=7.0cm,angle=0} }}
\vspace{-10pt}
\end{center}
\caption{Local schematics for a cylinder flip.}\label{f:figurea}
\end{figure}

Observe that if $a = b$ or if both $a$ and $b$ are contained in $\partial S$, then $S$ is exceptional. We can thus assume that $a \not\subset \partial S$ and $a\neq b$. In particular $a$ belongs to two distinct triangles on $u$ and we have the following pentagon as in Figure \ref{f:figureb}.

\begin{figure}[h]
\leavevmode \SetLabels
\L(.47*.48) $a$\\
\L(.59*.36) $b$\\
\L(.50*.005) $l$\\
\L(.55*.81) $l$\\
\L(.39*.29) $d$\\
\L(.42*.80) $c$\\
\endSetLabels
\begin{center}
\AffixLabels{\centerline{\epsfig{file =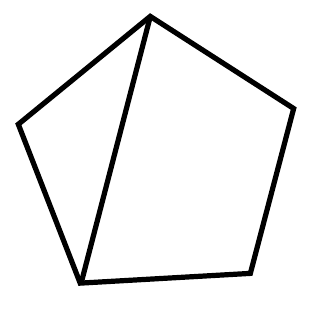,width=3.0cm,angle=0} }}
\vspace{-10pt}
\end{center}
\caption{A pentagon in $S$, with $\alpha$ belonging to two distinct triangles}\label{f:figureb}
\end{figure}
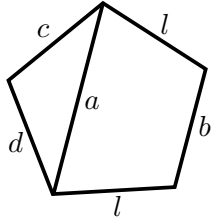


If $c=d$ then $b \not \subset \partial S$, for otherwise $S$ is exceptional. We can thus consider the other  triangle to which $b$ belongs, and denote its other edges by $e$ and $f$. Again, if $e=f$, then $S$ is exactly a four times punctured sphere and is thus exceptional. It follows that $e\neq f$. As the setup is symmetric, this in turn implies that $c\neq d$.
We therefore have three possibilities: 

\smallskip

\noindent{\em Case 1: $c,d \ne b$.} In this case, the original cylinder flip from $u$ to $v$ can be completed to a pentagon by considering the flips inside the pentagon bounded by $c,d, b$ and the two copies of $l$ (see Figure \ref{f:figurec}). The path we need between $u$ and $v$ is the path inside this pentagon that connects $u$ to $v$ in a clockwise manner. By the remark before the proof, the only remaining issue is to check that all the flips performed along this path are made on quadrilateral edges which are distinct in $S$. This may be verified, for instance, by looking at Figure \ref{f:figurec}.


\begin{figure}[h]
\leavevmode \SetLabels
\L(.51*.97) $\tiny{d}$\\
\L(.675*.725) $\tiny{d}$\\
\L(.475*.90) $\tiny{a}$\\
\L(.645*.65) $\tiny{a}$\\
\L(.457*.76) $\tiny{l}$\\
\L(.56*.83) $\tiny{l}$\\
\L(.602*.473) $\tiny{l}$\\
\L(.726*.575) $\tiny{l}$\\
\L(.43*.92) $\tiny{c}$\\
\L(.602*.67) $\tiny{c}$\\
\L(.52*.67) $\tiny{b}$\\
\L(.685*.415) $\tiny{b}$\\
\endSetLabels
\begin{center}
\AffixLabels{\centerline{\epsfig{file =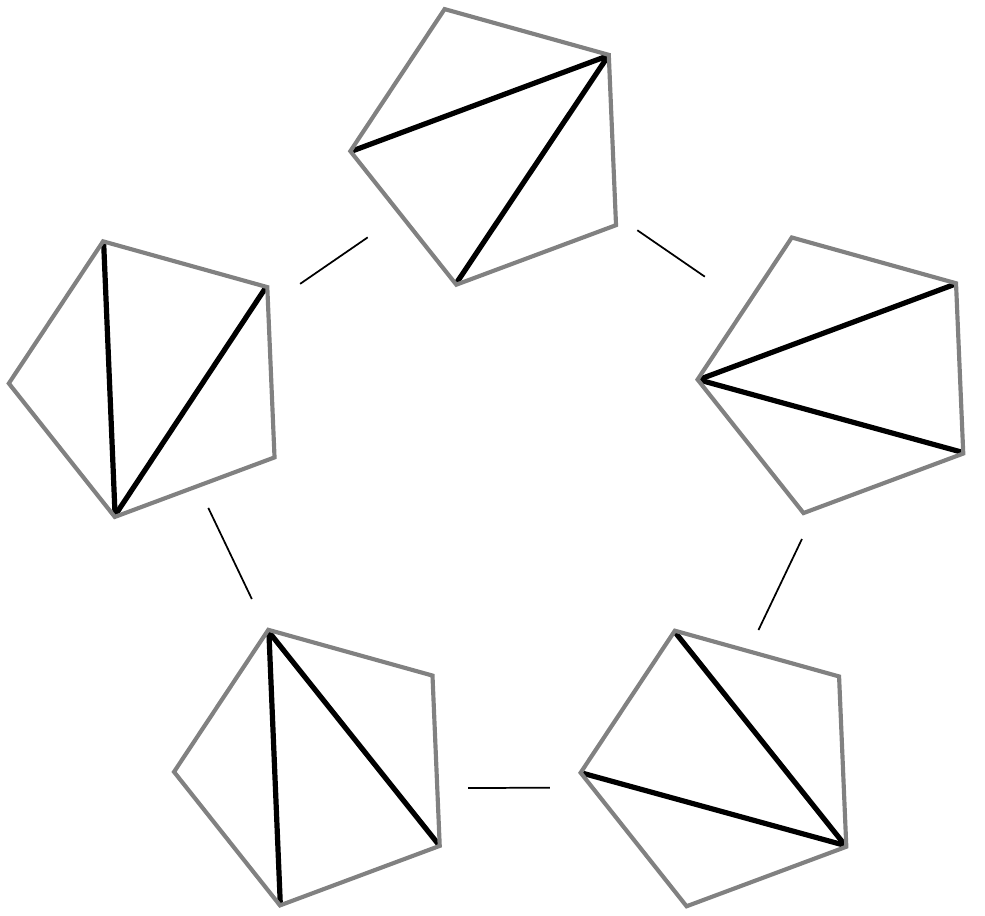,width=6.0cm,angle=0} }}
\vspace{-10pt}
\end{center}
 \caption{The pentagon in $\F(S)$ when $c,d \neq b$} 
\label{f:figurec}
\end{figure}

\smallskip

\noindent{\em Case 2: $d = b$.}  As  cannot be a torus with a single boundary component and a single marked point on the boundary, we have $c\not \subset \partial S$.  Therefore $c$ belongs to a second triangle; we denote the two remaining edges of this triangle by $c'$ and $c''$, as in Figure \ref{f:figured}. Again, as $S$ is not exceptional (and in this case is not a torus with two marked points), we have that that $c' \neq c''$. We now have a hexagon formed by the different edges of the triangles, and we flip in this hexagon following the schematics in Figure \ref{f:figuree}.

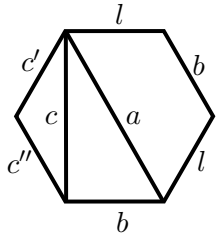
\begin{figure}[h]
\leavevmode \SetLabels
\L(.42*.45) $c$\\
\L(.505*.45) $a$\\
\L(.575*.7) $b$\\
\L(.58*.21) $l$\\
\L(.494*.96) $l$\\
\L(.495*-.1) $b$\\
\L(.395*.72) $c'$\\
\L(.38*.2) $c''$\\
\endSetLabels
\begin{center}
\AffixLabels{\centerline{\epsfig{file =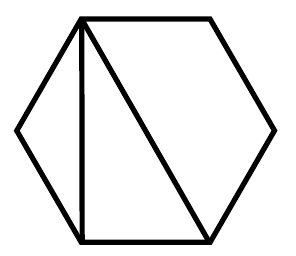,width=3.0cm,angle=0} }}
\vspace{-10pt}
\end{center}
\caption{A hexagon in $S$, where $d=b$} \label{f:figured}
\end{figure}

\begin{figure}[h]
\leavevmode \SetLabels
\L(.37*1.00) $l$\\
\L(.44*.80) $l$\\
\L(.44*.93) $b$\\
\endSetLabels
\begin{center}
\AffixLabels{\centerline{\epsfig{file =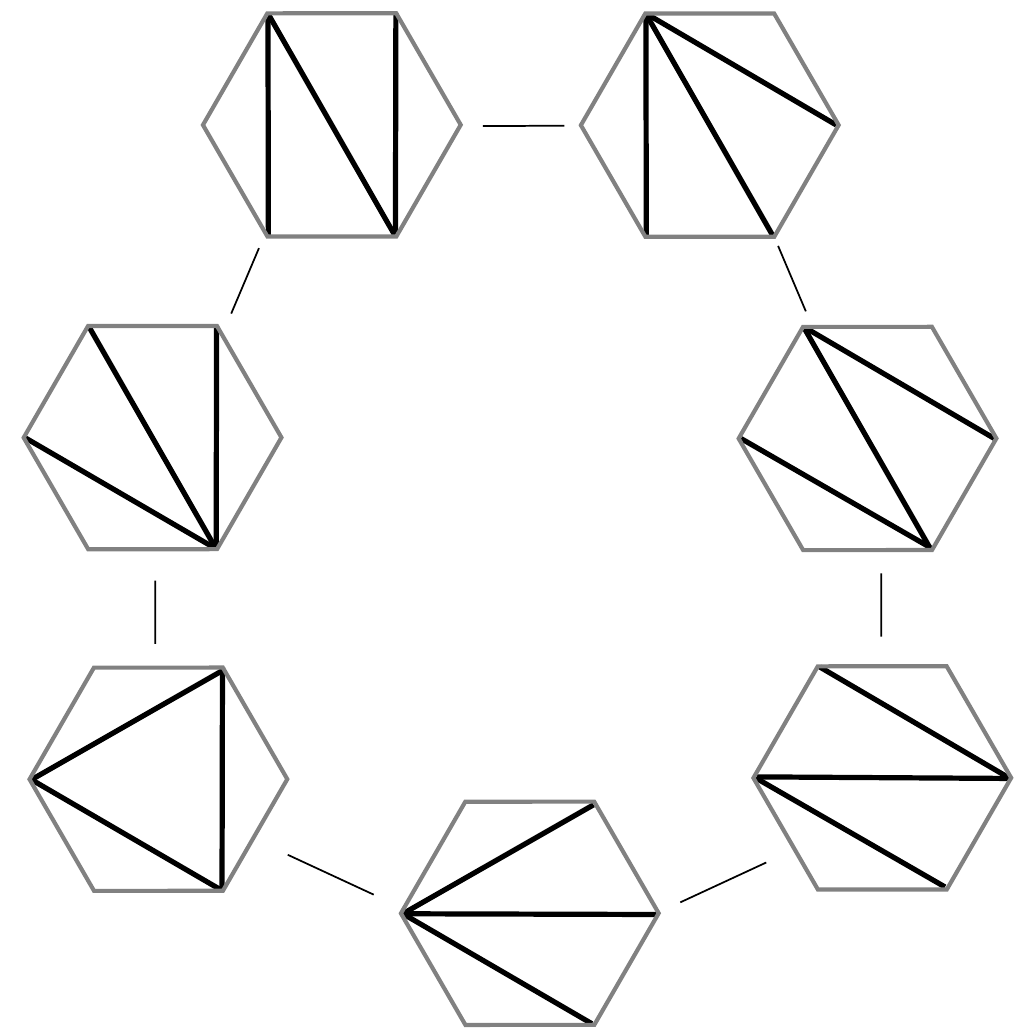,width=9.0cm,angle=0} }}
\vspace{-10pt}
\end{center}
\caption{Flipping inside a hexagon to get a path in $\F(S)$} \label{f:figuree}
\end{figure}


As before, it suffices to check that each flip is performed in a quadrilateral with four distinct edges in $S$, and we claim that this is the case.  The hexagon in $S$ under consideration has edges labeled cyclically by $\{d,l,d,l,c'',c'\}$.  We include the diagonal being flipped under the cylinder flip and obtain a sequence of flips as is illustrated by Figure \ref{f:figuree}.  The cylinder flip we wish to avoid is the top flip in the clockwise direction.  The sequence of flips we use to avoid the cylinder flip traverses the heptagon in the counterclockwise direction.  The claim follows immediately by inspection.

\smallskip

\noindent{\em Case 3: $c = b$.} We argue similarly -- this time $d \not  \subset \partial S$, so we consider the other triangle to which $d$ belongs.  As before, because $S$ is not exceptional, the two other sides of this triangle are necessarily distinct. Flipping in the resulting hexagon, we find a path (see Figure \ref{f:figuref}) in which it is easy to check again that all flips are performed in quadrilaterals with distinct sides.
This finishes the proof.\end{proof}

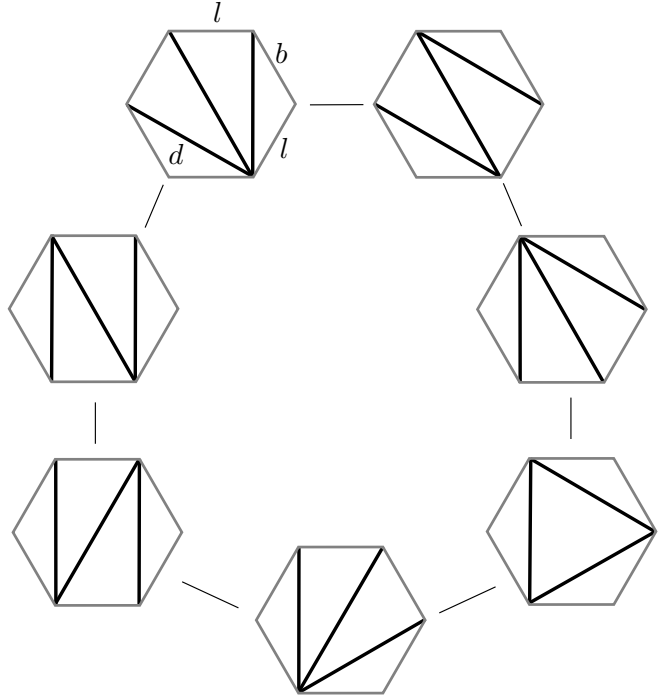
\begin{figure}[h]
\leavevmode \SetLabels
\L(.37*1.00) $l$\\
\L(.44*.80) $l$\\
\L(.435*.94) $b$\\
\L(.323*.79) $d$\\
\endSetLabels
\begin{center}
\AffixLabels{\centerline{\epsfig{file =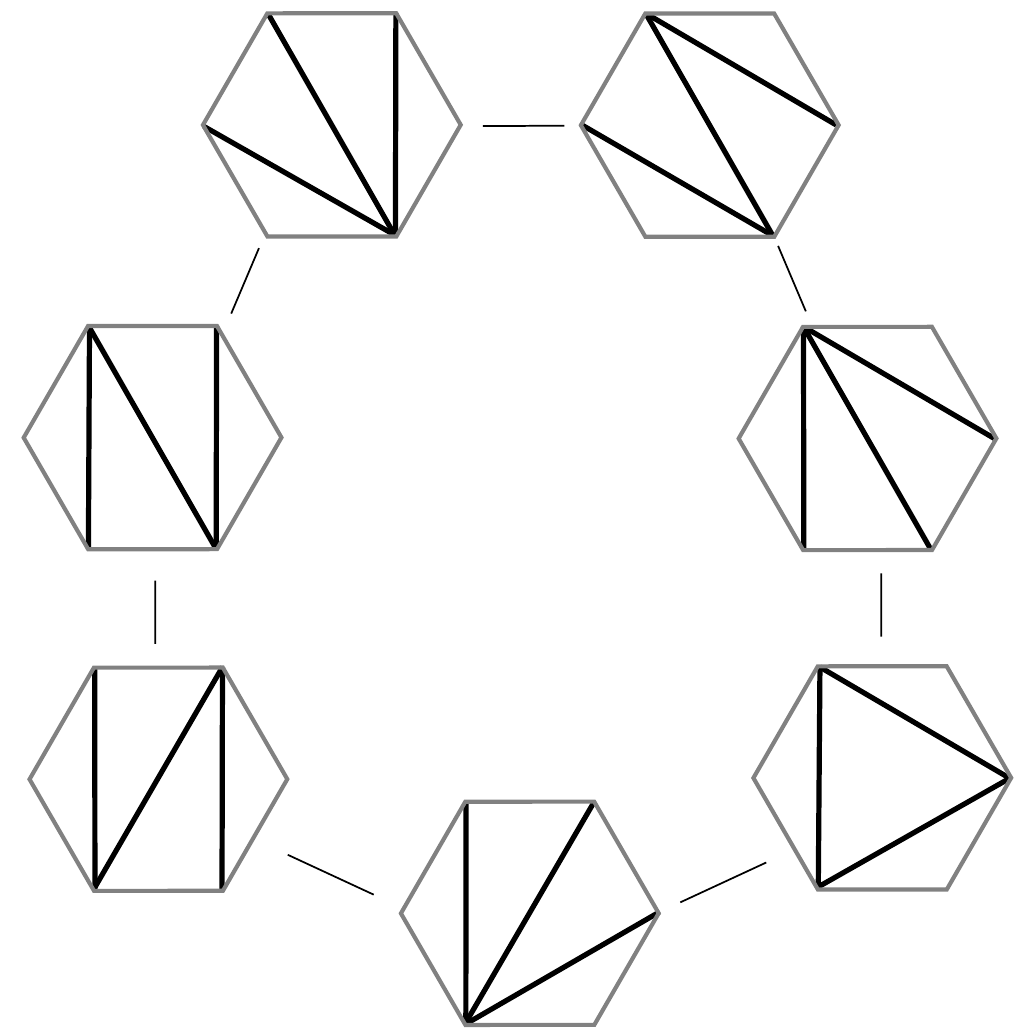,width=9.0cm,angle=0} }}
\vspace{-10pt}
\end{center}
\caption{Flipping in another hexagon to get a path in $\F(S)$} \label{f:figuref}
\end{figure}

We can now give proofs of the propositions. 

\begin{proof}[Proof of Proposition \ref{extend1}]

In light of Lemmas \ref{lem:degree} and \ref{lem:cylinder}, it suffices to show that any two vertices of maximal degree can be connected by a path $\gamma$, all of whose vertices have maximal degree. So take two vertices $u$ and $v$ of degree $d(S)$ and a  path $$\gamma_0: u=u_1, \ldots, u_n=v$$ of minimal length between them.  We will proceed by induction on $n$.

Suppose, for contradiction, that $\gamma_0$ contains a vertex of non--maximal degree, and let $1\le i\le n$ be the smallest index such that $\deg(u_i)<d(S)$. The flip from $u_{i-1}$ to $u_i$ has local schematics as illustrated in Figure \ref{f:figurebigi}.  We denote the loop surrounding the central vertex by $a$.

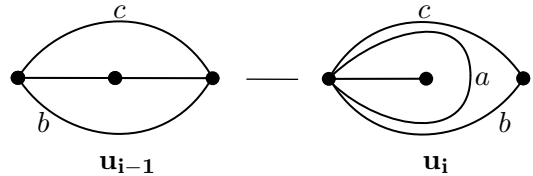
\begin{figure}[h]
\leavevmode \SetLabels
\L(.315*-.20) $\mathbf{u_{i-1}}$\\
\L(.655*-.20) $\mathbf{u_i}$\\
\L(.71*.45) $a$\\
\L(.25*.09) $b$\\
\L(.735*.09) $b$\\
\L(.33*.98) $c$\\
\L(.65*.98) $c$\\
\endSetLabels
\begin{center}
\AffixLabels{\centerline{\epsfig{file =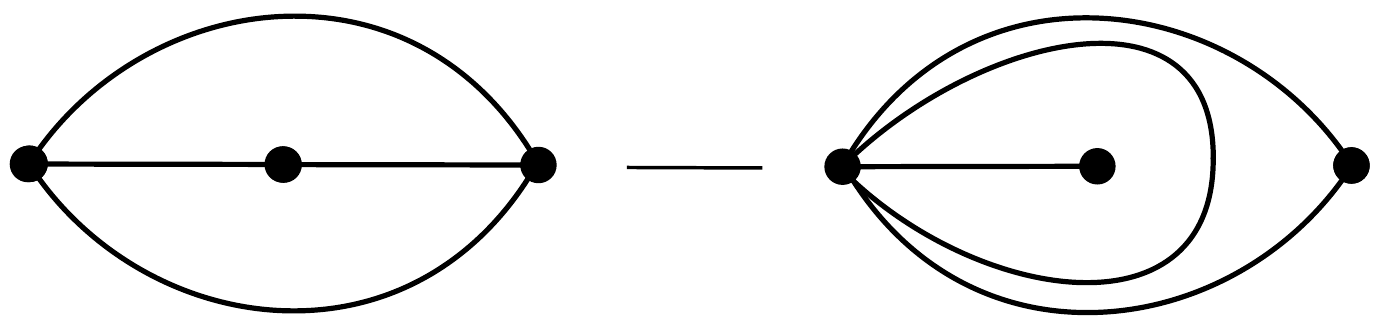,width=7cm,angle=0} }}
\vspace{-10pt}
\end{center}
\caption{Encountering a non--maximal degree vertex on $\gamma_0$} \label{f:figurebigi}
\end{figure}


As the degree of $v$ is maximal, the loop $a$ in Figure \ref{f:figurebigi} must be flipped again along $\gamma_0$. If the arc resulting from flipping $a$ remains surrounded by the two arcs $b$ and $c$ in Figure \ref{f:figurebigi}, then the result is locally described by the left side of Figure \ref{f:figurebigi}. However, the minimality of $n$ implies such a sequence of flips cannot occur, for otherwise $\gamma_0$ could be shortened to a path with the same endpoints and length $n-2$. In particular, at least one edge of $\gamma_0$ between $u_i$ and $v$ must flip one of the edges $\{b,c\}$. The local behavior of the resulting triangulation $u'$ obtained by flipping one of $\{b,c\}$ is given by Figure \ref{f:figurebigii}. 

\begin{figure}[h]
\leavevmode \SetLabels
\endSetLabels
\begin{center}
\AffixLabels{\centerline{\epsfig{file =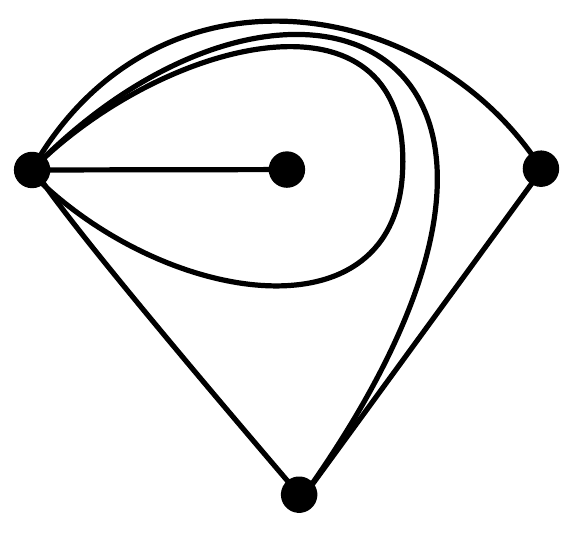,width=3.0cm,angle=0} }}
\vspace{-10pt}
\end{center}
\caption{The local behavior of the triangulation $u'$} \label{f:figurebigii}
\end{figure}

Observe that by construction, $d_{\F(S)}(v,u')<d_{\F(S)}(v,u_i)$, where $d_{\F(S)}$ denotes the combinatorial distance in $\F(S)$.  Now, Figure \ref{f:figurebigiii} gives a path to $u'$ which avoids introducing the arc $a$ until the last flip.

\begin{figure}[h]
\leavevmode \SetLabels
\endSetLabels
\begin{center}
\AffixLabels{\centerline{\epsfig{file =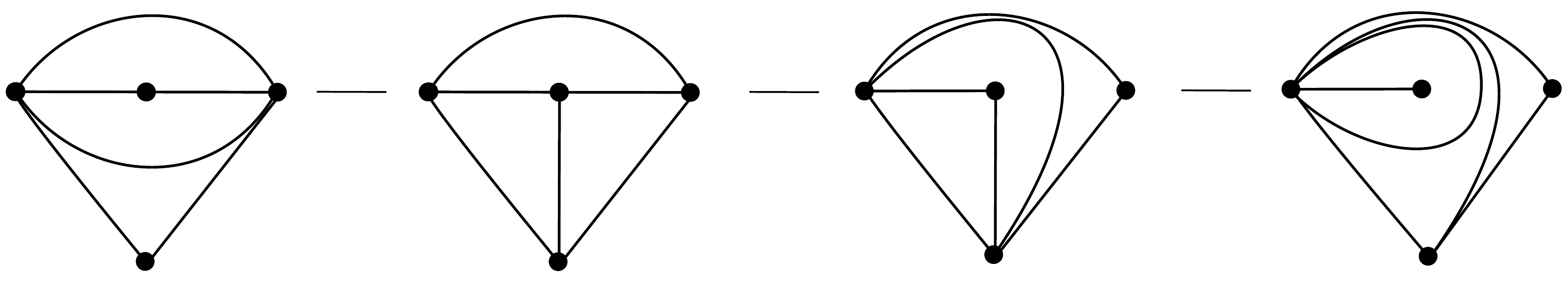,width=11cm,angle=0} }}
\vspace{-10pt}
\end{center}
\caption{Modifying $\gamma_0$ to avoid $a$ before $u'$} \label{f:figurebigiii}
\end{figure}

The previous two paragraphs furnish a path $\gamma_1$ between $u$ and $v$ with the following two properties: $u'$ is a vertex of $\gamma_1$ , and no triangulation corresponding to a vertex on $\gamma_1$ between $u$ and $u'$ contains the arc $a$.

We repeat the construction above, starting with the flip corresponding to the edge of $\gamma_1$ which resulted in the triangulation $u'$.  By induction on $n$, after $k\leq n$ modifications, we will obtain a path $\gamma_k$ from $u$ to $v$, each vertex of which other than the penultimate (i.e. the vertex immediately preceding $v$) does not contain the arc $a$.  From Figure \ref{f:figurebigi}, it follows that we may delete the last two edges of $\gamma_k$ to obtain a path from $u$ to $v$ in which every triangulation avoids $a$.


Because there are only finitely many possible unflippable arcs occurring along $\gamma_0$, and because none of the modifications performed to produce $\gamma_k$ from $\gamma_0$ introduce any new unflippable arcs, we may perform finitely many further modifications of $\gamma_k$ in order to obtain a path $\gamma$ from $u$ to $v$ in which each triangulation has no unflippable arcs.
\end{proof}

\begin{proof}[Proof of Proposition \ref{extend2}]
The proposition will follow from an argument similar to that of Proposition \ref{extend1}. Locally, the flip from $u$ to $v$ is modeled on Figure \ref{f:figurebigi}.

Suppose first $w \in \link(v) \setminus \{u\}$ is obtained by flipping along an arc $a$ which is not portrayed in Figure \ref{f:figurebigi}; in particular, $w = (v \setminus a) \cup a'.$
Then $v':= (u \setminus a) \cup a'$, $w$, $u$ and $v$ form a square in $\F(S)$ and we are done. 

Suppose now that $w \in \link(v) \setminus \{u\}$ is obtained by flipping along an arc which is portrayed in Figure \ref{f:figurebigi}. By symmetry we can suppose that it is the lower arc $b$ in such a configuration.  It follows that  the triangulation $w$ locally looks as in Figure \ref{f:figurebigii}.

Now by Lemma \ref{lem:cylinder} and Proposition \ref{extend1}, it suffices to show that we can connect $u$ and $w$ by a path whose every vertex has degree  $\deg(u)$. This is exactly one of the steps in the previous proof --- in fact, the path we need is the one shown in Figure \ref{f:figurebigiii}.
\end{proof}

\begin{remark} In fact, the arguments above show that, for every ``realizable" valence $d\le d(S)$, any two vertices of valence $d$ in $\F(S)$ may be joined by a path whose every vertex has valence at least $d$ and whose every edge is extendable. 
\end{remark}


\section{Proof of Theorem \ref{invariant-arc}}

\begin{proof}[Proof of  part (1) of Theorem \ref{invariant-arc}]
Choose  a vertex $u\in \F(S)$ with $\deg(u) = d(S)$. Since $\phi$ is injective, it follows that $$d(S) = \deg(u) \le \deg(\phi(u)) \le d(S'),$$ the desired conclusion. 
\end{proof}

\subsection{The invariant multiarc} The rest of the section is devoted to proving  part (2) of Theorem \ref{invariant-arc}. Let $u \in \F(S)$ be a vertex with $\deg(u) = d(S)$, and let $u_1, \ldots, u_{d(S)}$ be its neighbors. Define $$A(u) = \phi(u) \cap \phi(u_1) \cap \ldots \cap \phi(u_{d(S)});$$ that is, $A(u)$ is the collection of those arcs of $u$ that are not flipped when passing
from $u$ to $u_i$, with $i=1, \ldots, d(S)$. Observe that, by construction, $A(u) \subset \phi(v)$ for every $v\in \link(u)$. 
Since $\phi$ is injective, we have that $\phi(u_i) \ne \phi(u_j)$ if $i\ne j$, and thus $A(u)$ consists of exactly $d(S') -d(S)$ vertices.

Our first step is to prove the following: 

\begin{lemma}
Let $u,v \in \F(S)$ be vertices with $\deg(u) = \deg (v) = d(S)$. Then $A(u) = A(v)$.
\label{inv-lemma}
\end{lemma}

\begin{proof}
In light of Propositions \ref{extend1} and \ref{extend2}, it suffices to prove the result in the case where $u$ and $v$ are adjacent and where $(u,v)$ is an extendable edge. Suppose this is the case, noting
that $A(u) \subset \phi(v)$ by the construction of $A(u)$. It suffices to prove that $A(u) \subset \phi(w)$ for every $w$ adjacent to $v$. As the edge $(u,v)$ is extendable, then $u,v,w$ are contained in a square or pentagon in $\F(S)$, which we denote $\tau$.  Let $z$ be the vertex of $\tau$ adjacent to $u$ and not equal to $v$, noting that $A(u) \subset \phi(z)$, since $u$ and  $z$  are adjacent.  Now, the injectivity of $\phi$ implies that   $\phi(\tau)$ is a square or a pentagon in $\F(S')$.  It follows that the vertices of $\phi(\tau)$ have exactly $d(S') -2$ arcs in common, by Lemma \ref{shallow}. Now, $$A(u) \subset \phi(z) \cap \phi(u)\cap \phi(v) = \phi(z) \cap \phi(u)\cap \phi(v) \cap \phi(w),$$ so that $A(u)\subset \phi(w)$ as well, as we set out to prove. \end{proof}

Denote the multiarc associated to some (and hence any, by Lemma \ref{inv-lemma}) vertex of $\F(S)$ of valence $d(S)$ by $A$. As an immediate corollary of the proof of 
Lemma \ref{inv-lemma}, we have:

\begin{corollary}
Let $u\in \F(S)$ be a vertex such that $A \subset \phi(v)$ for every $v\in \link(u)$. If the edge $(u,v)$ is extendable, then $A \subset \phi(w)$ for every $w\in \link(v)$.
\label{holycrap}
\end{corollary}

We are finally in a position to prove part (2) of Theorem \ref{invariant-arc}. 

\begin{proof}[Proof of part (2) of Theorem \ref{invariant-arc}] Let $v$ be a vertex of $\F(S)$; we want to prove that $A \subset \phi(v)$. Choose a vertex $u$ of valence $d(S)$.  
By Propositions \ref{extend1} and \ref{extend2}, there exists a path $$u=u_0,u_1, \ldots, u_n=v$$ with the property that the edge $(u_i,u_{i+1})$ is extendable for all $0\le i< n$. 
Note that $A\subset \phi(w)$ for every $w\in \link(u)$, by Lemma \ref{inv-lemma} and the construction of $A$. 
Applying Corollary \ref{holycrap} and induction, we obtain that $A \subset \phi(u_i)$ for all $0\le i\le n$. This finishes the proof of Theorem \ref{inv-lemma}.
\end{proof}

\section{Proof of Theorem \ref{isom}}

We now give a proof of Theorem \ref{isom}. Let $S$ and $S'$ be connected surfaces, with $S$ non-exceptional and $d(S) = d(S') =d$. Let $$\phi:\F(S) \to \F(S')$$ be an injective simplicial map. The proof consists of several steps:

\subsection{Inducing a map on arcs} We first explain how the map $\phi$ induces a map $$\psi:\A(S) \to \A(S')$$ between the corresponding arc graphs.  Here, the arc graph of a surface $Z$ is the simplicial graph $\A(Z)$ whose vertices are arcs on $Z$, and where two arcs are adjacent in $\A(Z)$ if and only if they have disjoint interiors.   

Let $a \in \A(S)$ and consider $\F_a(S)$. 
Since $\F_a(S) \cong F(S\setminus a)$ and $d(S\setminus a) = d-1$, Theorem \ref{invariant-arc} applied to $$\phi: \F_a(S) \to \F(S')$$ implies that there exists a unique arc $b$ on $S$  such that $\phi(\F_a) \subset \F_b$. We set $$\psi(a) := b.$$ We state the next observation as a lemma, as we will need to make use of it later. In what follows, $i(\cdot,\cdot)$ denotes geometric intersection number between arcs.

\begin{lemma}\label{l:3} The  maps $\phi$ and $\psi$ satisfy the following properties:
\begin{enumerate}
\item  Let $v= (a_1, \ldots, a_d) \in \F(S)$. Then $\phi(v) = (\psi(a_1), \ldots, \psi(a_d))$.
\item Let $a\ne a' \in A(S)$ with $i(a,a') = 0$. Then $\psi(a)\ne \psi(a')$ and  $i(\psi(a),\psi(a')) =0$.
\item Let $a, a' \in A(S)$. If $i(a,a')=1$ then $i(\psi(a),\psi(a'))=1$.  
\label{relation}
\end{enumerate} 

\begin{proof}
Part (1) is an immediate consequence of the construction of $\psi$. For part (2), one may first extend  $a$ and $a'$ to a triangulation of $S$ and then apply (1). Finally, to see (3) note that there exist adjacent vertices $v,v'\in \F(S)$ such that $a\in v$ and $a'\in v'$. Since $\phi(v)$ and $\phi(v')$ are also adjacent, the result now follows from (1). 
\end{proof}

\end{lemma}
In the light of part (1) of Lemma \ref{l:3}, Theorem \ref{isom} will follow once we prove: 

\begin{proposition}
The map $\psi:\A(S) \to \A(S')$ is an isomorphism. 
\label{arcs}
\end{proposition}
 
Indeed, once Proposition \ref{arcs} has been established, Theorem \ref{isom} will follow as a combination of  Lemma \ref{l:3} (1) and Theorem A of Disarlo \cite{Disarlo},
which states that two arc graphs are isomorphic if and only if the underlying surfaces are homeomorphic; we remark that this result is originally due to Irmak-McCarthy \cite{IM}
in the particular case when $\partial S = \emptyset$ and $S'=S$.  
 
\subsection{Proof of Proposition \ref{arcs}} The  key ingredient in the proof of Proposition \ref{arcs} is the following:

\begin{lemma}\label{l:nonflip}
Let $v\in \F(S)$. If $a\in v$ is unflippable, then $\psi(a)\in \phi(v)$ is unflippable. 
\label{nonflip}
\end{lemma}

Before proving Lemma \ref{nonflip}, observe that if two arcs $a,a'$ are edges of the same triangle of a triangulation $v$, then $\psi(a)$ and $\psi(a')$ are edges of the same  triangle on $S'$. To see this, note that two arcs $a,a'$ are edges of the same triangle if and only if there exists an arc $b \in S$ such that $i(a,b) = i(a',b) = 1$ and $i(b,c)=0$ for all $c \in v\setminus  \{a, a'\}$. The desired conclusion now follows from  parts (2) and (3) of Lemma \ref{l:3}.

In particular, $\psi$ induces a map from the set triangles determined by $v$ to the set of triangles determined by $\phi(v)$. Moreover, if two triangles share  a given arc $a\in \A(S)$, then  their images under $\psi$ share the arc $\psi(a)$. 

\begin{proof}[Proof of Lemma \ref{l:nonflip}]
Let $v$ be a triangulation, and suppose $a\in v$ is unflippable, so  $v$ contains the arcs $a$ and $b$ depicted in Figure \ref{f:nonflip}, as we draw again in Figure \ref{f:nonflip2}. Extend $a$ and $b$ to a triangulation $v'$ that contains the solid arcs in Figure \ref{f:nonflip2}. Consider the triangulation $v''$ obtained from $v'$ by flipping $b$ to the arc $b'$ of Figure \ref{f:nonflip2}; as such, $v''$ determines two triangles $\Delta$ and $\Delta'$ that share exactly two arcs, namely $a$ and $b'$. 

By the preceding paragraphs and by Lemma \ref{l:3}, the triangles $\psi(\Delta)$ and $\psi(\Delta ')$ share exactly two arcs, namely $\psi(a)$ and $\psi(b')$.  Lemma \ref{l:3} implies that $\psi$ preserves the property of having intersection number 0 (resp. 1), so that we may conclude that $i(\psi(b),\psi(b'))=1$ and that $i(\psi(b),\psi(c))=0$ for every arc $c \in \Delta \cup \Delta'$ with $c\ne b'$.  In particular, $\psi(b)$ bounds a once punctured disk on $S'$ whose interior contains the interior of $\psi(a)$; in other words, $\psi(a)$ is unflippable with respect to $\phi(v')$, and hence is unflippable with respect to $\phi(v)$. 
\end{proof}

\begin{figure}[h]
\leavevmode \SetLabels
\L(.50*.190) $\tiny{b}$\\
\L(.45*.517) $\tiny{a}$\\
\L(.55*.517) $\tiny{b'}$\\
\endSetLabels
\begin{center}
\AffixLabels{\centerline{\epsfig{file =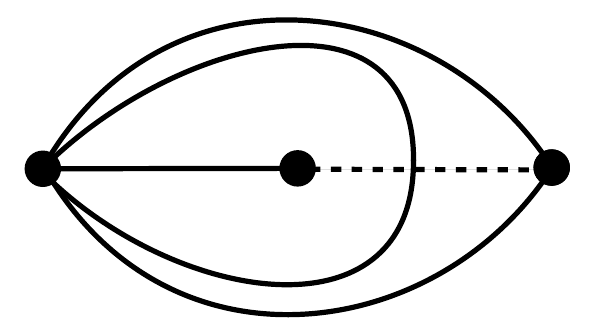,width=3.0cm,angle=0} }}
\vspace{-10pt}
\end{center}
\caption{An unflippable arc, revisited} \label{f:nonflip2}
\end{figure}

Observe that a vertex $v\in \F(S)$ has maximal valence if and only if all arcs of $v$ are flippable with respect to $v$. As a consequence of Lemma \ref{nonflip}, we obtain:

\begin{corollary}
For all $v\in \F(S)$, $\deg(\phi(v)) = \deg(v)$. In particular, $\phi$ is surjective and is thus an isomorphism. 
\label{degree}
\end{corollary}

\begin{proof}
For the first part, observe that Lemma \ref{nonflip} implies $\deg(\phi(v)) \le \deg(v)$; the other inequality is immediate, as $\phi$ is injective. 

To prove that $\phi$ is surjective,  let $w \in \F(S')$. By the connectivity of $\F(S)$, we may assume that $w$ is adjacent to $\phi(v)$, for some $v\in \F(S)$. Since $\deg(\phi(v)) = \deg(v)$, we have that $w=\phi(v')$ for some $v'$ adjacent to $v$, as desired. 
\end{proof}
We are finally in a position to prove Theorem \ref{isom}, which as we noted follows from Proposition \ref{arcs}:

\begin{proof}[Proof of Proposition \ref{arcs}]
First,  since $\phi$ is surjective, part (1) of Lemma \ref{l:3}  implies that $\psi$ is surjective as well. We claim  that $\psi$ is also injective. Let $a, b \in \A(S)$ with $a \ne b$. If $i(a,b)=0$ then $\psi(a)$ and $\psi(b)$ are distinct and disjoint, again by Lemma \ref{l:3} (1), and hence we are done. Thus assume $i(a,b)\ne 0$; equivalently, $\F_a(S) \cap \F_b(S) = \emptyset $. Using the fact that $\phi$ is an isomorphism, we obtain $$ \emptyset= \phi(\F_a(S)\cap \F_b(S)) = \phi(\F_a(S)) \cap \phi(\F_b(S)) = \F_{\psi(a)}(S') \cap \F_{\psi(b)}(S'),$$
and thus $i(\psi(a),\psi(b)) \ne 0$, which is what we set out to prove. 
\end{proof}

\section{Proof of Theorem \ref{main}}
\label{proofmain} We are now in a position to prove our main result. Suppose $S$ is a non-exceptional surface, and let $\phi:\F(S) \to \F(S')$ be an injective simplicial map. By Theorem \ref{invariant-arc}, there exists a multiarc $A$ on $S'$, with $d(S') -d(S)$ elements, such that $A \subset \phi(v)$ for every vertex $v$ of $\F(S)$. In other words, $\phi(\F(S)) \subset \F_A(S')$; recall that $\F_A(S')$ denotes the subgraph of $\F(S')$ spanned by those triangulations of $S'$ that contain $A$. 

Now, $\F_A(S') \cong \F(S'\setminus A)$, where $S' \setminus A$ is the surface obtained from $S'$ by cutting open along every element of $A$; note that $d(S) = d(S' \setminus A)$. The surface $S' \setminus A$ need not be connected; let  $\Sigma_1,\ldots, \Sigma_n$ be its connected components. By slight abuse of notation, the map $\phi$ induces a map $$\phi:\F(S) \to \F(\Sigma_1) \times \ldots \times \F(\Sigma_n).$$ 
Write $\pi_i:  \F(\Sigma_1) \times \ldots \times \F(\Sigma_n) \to \F(\Sigma_i)$ for the projection onto the $i$-th factor. 

\begin{claim}
Up to reordering of the indices, $\pi_i \circ \phi$ is trivial for all $i=2, \ldots, n$. 
\end{claim}

\begin{proof}[Proof of the claim]
Using the same notation as in the previous section, let $$\psi:\A(S) \to \A(\Sigma_1) \times \ldots \times \A(\Sigma_n)$$ be the map on arcs graphs induced by $\phi$, and choose $a\in \A(S)$. Up to reordering indices, we may assume  that $\psi(a) \subset \Sigma_1$. We claim that if $b \in \A(S)$ satisfies $i(a,b)=0$, then
$\psi(b) \subset \Sigma_1$ as well. Indeed, it suffices choose a third arc $c$ such that $i(a, c) = i(b,c)=1$, so that 
$i(\psi(a),\psi(c)) = i(\psi(b),\psi(c))=1$ by Lemma \ref{l:3}. 
The claim now follows from since $\A(S)$ is connected. 
\end{proof}

In the light of the claim above, we may view $\phi$ (again abusing notation) as a map $$\phi:\F(S) \to \F(\Sigma_1),$$ noting that $d(S) = d(\Sigma_1)$. Since $S$ is not exceptional and connected, it follows that $\phi$ is induced by a homeomorphism $S \to \Sigma_1$ by Theorem \ref{isom}. This finishes the proof of Theorem \ref{main}.

\end{document}